\documentclass{amsart}
\usepackage{pinlabel}
\usepackage{amsmath, amsthm, amssymb, amscd, mathrsfs, eucal, epsfig, color}
\usepackage{hyperref}

\input xy
\xyoption{all}

\begin{document}

\newtheorem*{theorem_A}{Theorem A}
\newtheorem{theorem}{Theorem}[section]
\newtheorem{lemma}[theorem]{Lemma}
\newtheorem{corollary}[theorem]{Corollary}
\newtheorem{conjecture}[theorem]{Conjecture}
\newtheorem{proposition}[theorem]{Proposition}
\newtheorem{question}[theorem]{Question}
\newtheorem*{answer}{Answer}
\newtheorem{problem}[theorem]{Problem}
\newtheorem*{main_theorem}{Main Theorem}
\newtheorem*{claim}{Claim}
\newtheorem*{criterion}{Criterion}
\theoremstyle{definition}
\newtheorem{definition}[theorem]{Definition}
\newtheorem{construction}[theorem]{Construction}
\newtheorem{notation}[theorem]{Notation}
\newtheorem{convention}[theorem]{Convention}
\newtheorem*{warning}{Warning}
\newtheorem*{assumption}{Simplifying Assumptions}

\theoremstyle{remark}
\newtheorem{remark}[theorem]{Remark}
\newtheorem{example}[theorem]{Example}
\newtheorem{scholium}[theorem]{Scholium}
\newtheorem*{case}{Case}

\def\id{\text{id}}
\def\H{\mathbb H}
\def\Z{\mathbb Z}
\def\N{\mathbb N}
\def\R{\mathbb R}
\def\C{\mathbb C}
\def\Q{\mathbb Q}
\def\CP{{\mathbb{CP}}}
\def\F{\mathcal F}

\def\length{\textnormal{length}}
\def\area{\textnormal{area}}
\def\rot{\textnormal{rot}}
\def\link{\textnormal{link}}
\def\wind{\textnormal{wind}}

\newcommand{\marginal}[1]{\marginpar{\tiny #1}}

\title{Taut Foliations leafwise branch cover $S^2$}
\author{Danny Calegari}
\address{Department of Mathematics \\ University of Chicago \\
Chicago, Illinois, 60637}
\email{dannyc@math.uchicago.edu}
\date{\today}

\begin{abstract}
A co-oriented foliation $\F$ of an oriented 3-manifold $M$ is taut if and only if there is a map
from $M$ to the 2-sphere whose restriction to every leaf is a branched cover.
\end{abstract}

\maketitle

\setcounter{tocdepth}{1}
\tableofcontents

\section{Introduction}

Let $M$ be a closed, oriented 3-manifold, and let $\F$ be an oriented, co-oriented
foliation of $M$ with 2-dimensional leaves. There are many equivalent definitions of
tautness for $\F$, including:
\begin{enumerate}
\item{there is a compact 1-manifold $Y$ (i.e.\/ a finite union of circles) 
transverse to $M$ such that every leaf intersects $Y$;}
\item{there is a closed 2-form $\omega$ which restricts to an area form on every leaf; or}
\item{no finite union of (positively oriented) closed leaves is homologically trivial.}
\end{enumerate}
See e.g.\/ \cite{Calegari_book} Chapter 4.

Here is another characterization of tautness, new as far as I know:
\begin{theorem}[Branched Covering]\label{Theorem_A}
An oriented/co-oriented foliation $\F$ of a 3-manifold $M$ is taut if and only if
there is a map $\phi:M \to S^2$ whose restriction to every leaf is a branched cover.
\end{theorem}
This means that locally, the restriction of $\phi$ to each leaf looks
like $z \to z^n$ for some $n>0$ (after perturbation, we may take $n=2$).
Note that taut foliations $\F$ cannot necessarily 
be assumed to be smooth, though we can always assume
(after isotopy) that the individual {\em leaves} are smooth 
(see e.g.\/ \cite{Calegari_smooth} or \cite{Kazez_Roberts}) and we can insist
that $\phi$ as above is smooth.

\begin{proof}
Suppose there is such a map $\phi$.
The branch locus of $\phi$ is discrete in each leaf, and after adjusting $\phi$ by
a small perturbation we may assume that every branch point has local degree 2. We call such branch
points {\em simple}; they correspond to simple zeros of $d\phi$ where we think of
$\phi$ as a (leafwise) holomorphic function of one complex variable. If branch points
are all simple, the set of branch points sweeps out a compact 1-manifold $X$ 
transverse to $\F$. Say that $\phi:(M,\F) \to S^2$ is in {\em general position} 
under such a circumstance.

We claim that either $M=S^2\times S^1$ foliated as a product by $S^2\times\text{point}$,
or the branch locus $X$ intersects every leaf of $\F$. In either case it follows that
$\F$ is taut. Here is the proof.
Let $\lambda$ be a leaf. Since $M$ is compact, the map $\phi$ is bilipschitz on leaves away
from the $\epsilon$-neighborhood of the branch locus. So if $\lambda$ were disjoint
from $X$, the map $\phi:\lambda \to S^2$ would be a covering map, and therefore we would
have $\lambda=S^2$. 
The Reeb Stability Theorem (see \cite{Calegari_book} Thm.~4.5) says that if $\F$ has a spherical
leaf, every leaf is a sphere and (assuming $M$ is orientable) $\F$ is the product foliation. 

\medskip

The harder direction is to construct the map $\phi$ if $(M,\F)$ is taut. I know
of two analytic methods that accomplish this, neither of them trivial:

\medskip

{\bf Method 1:} Ghys \cite{Ghys} Thm.~7.5 shows that if $\F$ is a taut foliation with all 
leaves conformally hyperbolic, there is a complex line bundle $L$ over $M$ which is
leafwise holomorphic and which admits many leafwise holomorphic sections. The ratio 
of two such sections defines a map from $M$ to $\CP^1$ which is leafwise
holomorphic and nonconstant, and therefore a leafwise branched cover.

\medskip

Ghys' argument depends on work of Candel \cite{Candel} and 
generalizes Poincar\'e's construction of holomorphic functions via hyperbolic geometry.
Unfortunately it does not seem to apply directly when $\F$ has leaves of mixed conformal
type. However there is a second method which does apply in this case:

\medskip

{\bf Method 2:} Donaldson \cite{Donaldson} shows that if $(W,\omega,J)$ is a symplectic manifold
with integral periods and a compatible almost-complex structure $J$ then $W$ 
has the structure of a Lefschetz pencil --- some positive multiple $n\omega$ is the first Chern
class of a line bundle $L$ which has many `almost' holomorphic sections, and the
ratio of two such sections defines a map from $W$ to $\CP^1$ which is `almost' 
(depending on $n$) holomorphic on $J$-holomorphic curves.

Now, Eliashberg--Thurston \cite{Eliashberg_Thurston} (for $C^2$ foliations)
and Bowden \cite{Bowden} (in general) show that if $(M,\F)$ is taut then
either $M=S^2\times S^1$ foliated by spheres (in which case the proof is immediate), 
or there is a symplectic almost-complex structure on $M\times [-1,1]$ with
pseudo-convex boundary, for which the foliation $\F$ of $M\times 0$ is $J$-holomorphic.
Independently Eliashberg \cite{Eliashberg} and Etnyre \cite{Etnyre} showed that
$M\times [-1,1]$ as above can be filled to a closed symplectic manifold $(W,\omega,J)$. 
Now apply Donaldson's theorem.
\end{proof}
Though elegant, neither analytic argument is completely satisfactory. Ghys'
technique does not apply in full generality, and Donaldson's technique goes via a rather
convoluted path. Moreover, both constructions depend on taking some parameter
(roughly speaking, the first Chern class of an appropriate line bundle) sufficiently
large, and there is no {\it a priori} control over how large is `sufficient'.
Finally, the characterization of tautness in Theorem~\ref{Theorem_A} is basically
combinatorial, and one would like to have a purely combinatorial construction.
Therefore we shall give two direct combinatorial constructions of $\phi$ in
\S~\ref{combinatorial}, both of which can be made effective, in a certain sense.

\section{Examples}\label{finite_depth}

First we give explicit constructions of $\phi$ in some simple special cases.

\subsection{Surface bundles}

The simplest class of taut foliations are surface bundles. Let $R$ be a closed oriented
surface, and let $\psi:R \to R$ be an orientation-preserving diffeomorphism. The mapping
torus $M_\psi$ is foliated by copies of $R$ (i.e.\/ fibers of the obvious fibration over
$S^1$), and this foliation is evidently taut.
We shall construct a map $\phi:M_\psi \to S^2$ which restricts on each fiber to
a degree $d$ branched covering map with $4(d-1)$ simple branch points. Here's how to
do it. 

The data of a degree $d$ branched cover $R \to S^2$ with $n$ simple branch points is 
given by the images $P_i$ of the branch points, and a connected degree $d$ cover $R'$ of
$S^2 - \cup_i P_i$ whose monodromy around each $P_i$ is a transposition $\sigma_i$. 
To define a cover, the product of the $\sigma_i$ must be the identity, and for the
cover to be connected, the $\sigma_i$ must generate a group acting transitively on
the $d$-points. We compute $\chi(R)=2d-n$ so if $n=4(d-1)$ the genus of $R$ is $(d-1)$.

Explicitly, let's partition the $4(d-1)$ points $P_i$ into $(d-1)$ sets of
consecutive 4 indices $1234, 5678, \cdots$ and let $\sigma_j$ be the transposition
that permutes $i$ and $i+1$ for $j$ in the $i$th set of 4 indices
$4i-3 \le j\le 4i$. Then $\prod_i \sigma_i=\id$ and the action is evidently transitive.

Each 4-tuple of indices contributes 1 handle to $R$; braiding each set of 4 points
in $S^2$ locally effects Dehn twists in the meridians and longitudes of these
handles on $R$. If $4i-3\cdots 4i, 4i+1 \cdots 4i+4$ is an adjacent pair of 4-tuples, then
braiding $4i-1\, 4i$ around $4i+1\, 4i+2$ effects a Dehn twist along the curve linking
a pair of adjacent handles on $R$. As is well known, this set of Dehn twists generates
the mapping class group of $R$, so we can realize the monodromy $\psi$ by a suitable
braid of $4(d-1)$ points in $S^2$.

\subsection{Finite depth foliations}

If $\F$ is a (taut) foliation, a leaf $\lambda$ has {\em depth 0} if and only if it is closed, 
and {\em depth $\le k$} if and only if the closure $\overline{\lambda}$ is equal to the union of 
$\lambda$ together with leaves of depth $\le k-1$. The foliation has depth $\le k$ if
every leaf has depth $\le k$, and is {\em finite depth} $k$ if it has depth $\le k$ for 
some least $k$. Thus, a foliation has depth 0 if and only if the leaves are the fibers of
a fibration over the circle.

Denote the union of leaves of depth $\le k$ by $\F_k$. This is a closed set.
If $\F$ has finite depth $k$, we can always adjust the foliation so that there are
finitely many leaves of depth $<k$, and the complement of the closed set $\F_{k-1}$ is
an open manifold which has the structure of a surface bundle over $S^1$ with fiber a 
(typically noncompact, and possibly disconnected) leaf of depth $k$.

Let's consider the depth $1$ case. $\F_0$ is a finite set of closed surfaces, and the
complementary regions fiber over $S^1$. Let $U$ be a complementary region, i.e.\/ a surface
bundle over $S^1$ with noncompact fiber. The fiber
$\lambda$ is proper in $U$, so that it has finitely many ends, each of which spirals
around a leaf of $\F_0$. In other words, the ends look like infinite cyclic covers
of closed surfaces, and the monodromy $\psi$ restricts on each end to the deck
transformation of this cover.

We can write $\lambda$ as $\lambda:= \lambda_+ \cup \lambda_0 \cup \lambda_-$ as follows.
Let $R$ and $L$ be the (possibly disconnected) compact surfaces that the ends of
$\lambda$ spiral around in the positive and negative direction, 
and let $\overline{R}:= \cup_{i \in \Z} R_i$ and $\overline{L}:=\cup_{i \in \Z}
L_i$ be infinite cyclic covers, expressed as a union of compact subsurfaces $R_i$ and $L_i$
so that the action of the deck group is $R_i \to R_{i+1}$ and $L_i \to L_{i+1}$. Then
$\lambda_+ = \cup_{i>0} R_i$ and $\lambda_- = \cup_{i<0} L_i$, and $\psi$ acts on
$\lambda_+$ and $\lambda_-$ like the restriction of the deck action of the cover, glued
together along their boundaries by some homeomorphism 
$\psi_0: L_{-1} \cup \lambda_0 \to \lambda_0 \cup R_1$.
Note that $\chi(L_i)=\chi(R_j)$ is independent of $i$ and $j$.

A branched cover $\phi_R:R \to S^2$ branched over a subset $X$ lifts to a deck group
invariant map $\phi_{\overline{R}}:\overline{R} \to S^2$
branched over finite subsets $X_i \subset R_i$ permuted by the deck group, and similarly
we can get maps $\phi_{\overline{L}}$ 
and branch points $Y_i \subset L_i$, where we choose maps with the same
degree so that $X_i$ and $Y_i$ have the same cardinality. Extend $L_{-1} \to S^2$ branched over
$X_{-1}$ to some $\phi_0:L_{-1} \cup \lambda_0 \to S^2$ branched over $X_{-1} \cup Z$ for a finite
subset $Z \subset \lambda_0$, compatible with the obvious maps on the boundary. Restricting
$\phi_0$ to $\lambda_0$ and extending it over $R_1$ determines another branched map
$\phi_0':\lambda_0 \cup R_1 \to S^2$ branched over $Z \cup Y_1$, and we can compare the two
maps $\phi_0,\phi_0'\psi_0:L_{-1} \cup \lambda_0 \to S^2$. 

Any two simple branched maps 
from compact surfaces to $S^2$ of the same (sufficiently large) degree are isotopic through
such maps, and we can extend this isotopy end-periodically to an isotopy 
between branched maps $\phi,\phi\psi:\lambda \to S^2$.
The suspension of this isotopy determines a leafwise branched map 
$\phi$ from $U$ (the mapping torus of $\lambda$) to $S^2$ which extends continuously 
to the given maps $\phi_R,\phi_L$ on the boundary leaves $R$ and $L$. Performing this
extension on each complementary component defines a leafwise branched cover 
$(M,\F) \to S^2$.

\section{Two combinatorial proofs of the Branched Covering Theorem}\label{combinatorial}

In this section we give two combinatorial proofs of Theorem~\ref{Theorem_A}, valid in
full generality. The proofs are completely explicit and constructive, 
though profligate in a certain sense, which will be explained in \S~\ref{realization}.
Both proofs make direct use of the combinatorial definition of tautness as the
existence of a transverse 1-manifold $Y$ through every leaf.

\subsection{Triangulations} 

Several well-known constructions in the theory of foliations make use of an auxiliary
triangulation; for example, Thurston's well-known 
construction \cite{Thurston} of a foliation on a 3-manifold
in any homotopy class of 2-plane field is of this kind. Thurston constructs the foliation
locally over the 2-skeleton, but there is an obstruction to extending it naively over
the 3-simplices. This obstruction is overcome by drilling out a family of transverse loops
which bust through each 3-simplex and are transverse to the part of the foliation
constructed so far, solving the (now unobstructed) problem on 3-simplices, and `turbularizing'
the result over the drilled out loops (this creates Reeb components, which is unavoidable
in the context of Thurston's construction). Our first construction
bears a close family resemblance to Thurston's argument.

\begin{proof}
First, choose a triangulation $\tau$ of $M$ in such a way that the restriction of $\F$
to each simplex is (topologically) conjugate to the foliation of a linear simplex in 
$\R^3$ in general position. The existence of such a triangulation is easy: just work in
sufficiently small charts so that the foliation is nearly affine. Note that every edge
of the triangulation is transverse to $\F$, so that the co-orientation of $\F$ induces
an orientation on every edge compatible with a total order on the vertices of each
simplex. 

We can easily find a map $\phi:N(\tau^2) \to S^2$ from a neighborhood of the 2-skeleton to
$S^2$ which is actually an {\em immersion} when restricted to each leaf of $\F|N(\tau^2)$.
To see this, observe that the extension problem for $\phi$ over an $i$-simplex rel.
a choice of $\phi$ on a neighborhood of its boundary amounts to 
taking a solid cylinder $D^{i-1}\times I$ foliated by $D^{i-1} \times \text{point}$ and 
extending $\phi: \partial (D^{i-1}\times I)\to S^2$ to the interior so that the resulting
extension is an immersion on each horizontal $D^{i-1}\times \text{point}$ 
(the extension to a neighborhood of the simplex
in the 3-manifold is then obtained e.g.\/ by exponentiating the normal bundle of the image
leafwise). For $i\le 2$ the extension is easy: the $i=1$ case is just the problem of
joining two points in $S^2$ by a path, while the $i=2$ case is to find a
1-parameter family of immersed paths in $S^2$ with a specified family of boundary values 
interpolating between two given immersed paths. 
The latter problem can be solved, because the space of immersed paths
in $S^2$ joining any two points is nonempty and path-connected. 

The trouble comes only when we try to extend $\phi$ over the 3-simplices. Let's state the
extension problem explicitly. We have a map $\phi$ defined on the boundary of a solid cylinder
$\phi:\partial (D^2\times I) \to S^2$ so that the maps $\phi_i:D^2 \times i\to S^2$ on the top
and bottom disks are both immersions, and 
so are the 1-parameter family of maps $\partial \phi_t:\partial D^2\times t\to S^2$,
and we would like to find immersions $\phi_t:D^2 \times t$ extending $\partial \phi_t$
and interpolating between $\phi_0$ and $\phi_1$. No such extension exists in general.

\medskip

Here is where we use the hypothesis of tautness. For each simplex $\sigma$ we will find
a loop $Y_\sigma$ transverse to $\F$ with a certain property. We want $Y_\sigma$ to be
the union of two arcs transverse to $\F$. The first is an arc $\alpha$ in the interior of
$\sigma$ that runs from the disks near the top and bottom vertex where $\phi$ has been
defined. The second is an arc $\beta$ completely contained in the neighborhood $N(\phi)$ 
of the 2-skeleton where $\phi$ is already defined. It is easy to find such a $\beta$: by 
tautness, any $\alpha$ can be completed to a loop $\alpha \cup \beta'$ transverse to $\F$; 
now homotop $\beta'$ leafwise in each simplex until the result $\beta$ is contained in 
(a neighborhood of) the 2-skeleton. 

We now {\em plumb} the map $\phi$ along $\beta$. What does this mean? A neighborhood $N(\beta)$
of $\beta$ is foliated as $\text{disk} \times I$, and $\phi$ maps each disk in the product
to $S^2$ by an embedding. We modify $\phi$ on $N(\beta)$, keeping it fixed near the boundary,
by replacing each embedding $D \to S^2$ by a branched map $D \to S^2$ with $2n$ simple
branch points, so that relative to the embedding, the new map has some large fixed
(positive) degree $n$.

The net result is to change the boundary values of our extension problem: we have a new
map $\phi:\partial (D^2\times I) \to S^2$ so that $\phi_i:D^2 \times i \to S^2$ are
branched immersions of (positive) degree $n$, and a 1-parameter family of immersions
$\partial \phi_t:\partial D^2 \times t \to S^2$ interpolating between $\partial \phi_i$ for
$i=0,1$. Any fixed 1-parameter family of immersed loops in $S^2$ bounds a
family of immersed branched disks $D \to S^2$ of degree $n$ whenever $n$ is large enough,
so we can extend the modified $\phi$ over the interior of $\sigma$. Performing this extension
simplex by simplex, we obtain the desired map $\phi$.

Another way to say this (more in line with Ghys' perspective) is that the plumbing
operation increases the relative first Chern class $c_1$ for the boundary maps $\phi_i$
of disks to $S^2$, until the family $\partial \phi_t$ 
has a holomorphic extension over $D$. 
\end{proof}

\begin{remark}
If $\xi$ is a contact structure on $M$, the failure of integrability implies that
we can construct (locally) a circle in $M$ through any given point and transverse to
$\xi$. The construction above works essentially without change to give a 
branched map $\phi:(M,\xi) \to S^2$ transverse to {\em any} (co-orientable)
contact structure.

As mentioned earlier, Eliashberg--Thurston \cite{Eliashberg_Thurston} and Bowden \cite{Bowden}
show that if $\F$ is {\em any} co-orientable foliation of an orientable 3-manifold $M$,
then unless $M$ is $S^2 \times S^1$ or $M=T^3$ (and $\F$ is a foliation by planes),
then $\F$ can be $C^0$ approximated by positive and negative contact structures.

We can construct branched maps to $S^2$ for these approximating contact structures. 
However, these maps cannot in general be chosen in such a way as to limit to a 
leafwise branched map $\phi:(M,\F) \to S^2$ since $\F$ might not be taut.
\end{remark}

\subsection{Belyi maps}\label{Belyi_section}

If $R$ is a closed Riemann surface, a {\em Belyi map} $R \to \CP^1$ is a holomorphic map
branched only over the points $0,1,\infty$. It is a rather beautiful theorem of Belyi
that a Riemann surface may be defined by equations over $\overline{\Q}$ 
if and only if it admits a Belyi map.

If $R \to \CP^1$ is a Belyi map, we may decompose $\CP^1$ into two triangles 
(one black one white) and vertices at $0,1,\infty$, and then pull back to obtain
a triangulation of $R$.

Conversely, let $\tau$ be any triangulation $\tau$ of an oriented surface $R$
and let $\tau'$ be its first barycentric subdivision. For every triangle $\Delta$ of
$\tau'$ we can canonically order its vertices $0,1,2$ according to the
dimension of the simplex of $\tau$ they are at the center of. $\tau'$
has a canonical 2-coloring, according to whether the orientation of each
triangle $\Delta$ coming from this ordering agrees or disagrees with 
the orientation on $R$. Thus $R$ admits a canonical Belyi map to $\CP^1$ 
associated to $\tau'$.

Here is another combinatorial proof of Theorem~\ref{Theorem_A}, via Belyi maps.

\begin{proof}
Assume $\F$ is not the product foliation of $S^2 \times S^1$ by spheres (or else
the conculsion is obvious). Thus no leaf of $\F$ is a sphere.

Let $Y$ be a total transversal for $\F$. We want to define a certain kind of
combinatorial structure on $M$ compatible with $Y$. This structure is a leafwise
triangulation --- i.e.\/ a triangulation of each leaf --- with vertices precisely at
$Y$, and varying continuously transverse to $\F$ except at finitely many disjoint
`squares' $S$ contained in leaves of $\F$ where two triangulations are 
related by a $2$--$2$ move (i.e.\/ the move which switches the two diagonals of 
the square $S$).

It is probably easy to construct such a structure by hand, but it is convenient
to use a geometric argument. 
To show that such a structure exists, first observe that we can find a metric on
leaves of $\F$, continuously varying in the transverse direction, for which 
each leaf $\lambda$ has the structure of a complete hyperbolic orbifold 
with cone angle $\pi$ at the points of $\lambda \cap Y$. This is true leafwise in
any conformal class (because $\F$ has no spherical leaves and any noncompact leaf
intersects $Y$ in infinitely many points) and Candel \cite{Candel} shows that
leafwise uniformization is continuous (his argument applies without change to the
good orbifold case). 

Now for each leaf $\lambda$ we obtain a canonical triangulation with vertices at
$\lambda \cap Y$ by taking the dual of the Voronoi tessellation with centers at 
$\lambda \cap Y$. This is generically a triangulation, except at finitely many
isolated squares corresponding to 4-valent vertices of the Voronoi tessellation.
This proves the existence of the desired combinatorial structure, as claimed.

For each triangulated leaf, barycentrically subdivide and map to $S^2$ by the
canonical Belyi map. At the $2$--$2$ transitions we have two different 
Belyi maps of degree 6 on a square. We `cut open' $M$ along a neighborhood 
of this square, and insert an isotopy between these two Belyi maps through 
(non-Belyi) branched maps. This gives the desired map.
\end{proof}

\section{Homological Invariants}

A homotopy class of map $\phi:M \to S^2$ has certain homological invariants 
attached to it. We give formulae for these invariants for $\phi:(M,\F) \to S^2$
in terms of data associated to the geometry and topology of $\F$.

\subsection{Euler class}

Associated to an oriented/co-oriented taut foliation $\F$ of $M$ is the Euler
class $e(T\F)\in H^2(M)$. If $\phi:(M,\F) \to S^2$ is any map, we can pull back
the Euler class $e(S^2):=e(TS^2) \in H^2(S^2)$ under $\phi^*$ to get another
class $\phi^*e(S^2) \in H^2(M)$. Let's assume that $\phi$ is a generic leafwise branched
cover, and let 
$X$ be the (simple) branch locus, oriented by the co-orientation of $\F$. 
If $[X] \in H_1(M)$ denotes the homology class of $X$, let $[X]^* \in H^2(M)$ denote
its Poincar\'e dual. Then these classes are related as follows:

\begin{theorem}[Euler class formula]\label{Euler_formula}
Let $(M,\F)$ be oriented and co-oriented, let $\phi:(M,\F) \to S^2$ be a generic
leafwise branched cover, and let $X$ be the (simple) branch locus. Then
$$e(T\F) + [X]^* = \phi^*e(S^2)$$
Consequently $[X] \cap [\mu] > -e(T\F)[\mu]$ for every invariant tranverse measure $\mu$.
\end{theorem}
\begin{proof}
The class $e(S^2)$ is the obstruction to finding a nonzero vector field on $S^2$ and
the class $e(T\F)$ is the obstruction to finding a nonzero vector field on $T\F$. Choose 
a vector field $V$ with index 1 at the north and south poles $n,s \in S^2$ where if necessary
we perturb $\phi$ so that $\phi(X)$ is disjoint from $n$ and $s$. The preimage 
$N \cup S:=\phi^{-1}(n\cup s)$ is Poincar\'e dual to $\phi^*e(S^2)$. Furthermore, the
preimage of $V$ defines a section of $T\F$ that vanishes at $N\cup S \cup X$. This
proves the formula.

After smoothing the pullback of the area form of $S^2$ we obtain a closed 2-form
$\omega$ in the class of $\phi^*e(S^2)$ positive on $T\F$, and such a class must pair
positively with every invariant transverse measure $\mu$. This proves the inequality.
\end{proof}

We interpret the inequality $[X]\cap [\mu]> -e(T\F)[\mu]$. For this we must briefly
recall some facts from the theory of foliation cycles; see Sullivan \cite{Sullivan_cycles}.
It is common to use the abbreviation $\chi(\mu):=e(T\F)[\mu]$; 
for $\mu$ an atomic measure of mass 1 supported
on a closed leaf $R$, this is equal to the Euler characteristic $\chi(R)$.

Any invariant transverse  measure decomposes into ergodic components. 
If $\mu$ is an ergodic invariant transverse 
measure, then for any leaf $\lambda$ in the support of $\mu$ there is an exhaustion by compact
sets $D_i \subset \lambda$ with $\length(\partial D_i)/\area(D_i) \to 0$ so that the
de Rham currents $D_i/\area(D_i)$ converge to a positive multiple of $\mu$.

An invariant transverse measure $\mu$ has $\chi(\mu)>0$ if and only if some positive
measure of leaves are 2-spheres; in this case by Reeb stability the foliation is the
product foliation of $S^2\times S^1$. Otherwise, if some $\mu$ has $\chi(\mu)=0$ the support 
of the measure is conformally parabolic; since $M$ is taut, 
this implies that $M$ is necessarily toroidal. 

Thus if $M$ is atoroidal (i.e.\/ hyperbolic) $\chi(\mu)<0$ for every $\mu$. By
Candel \cite{Candel} there is a metric on $M$ (in every conformal class) for which leaves
of $\F$ have constant curvature $-1$. 
Let's rescale our invariant transverse measure so that $D_i/\area(D_i) \to \mu$. 
Then by Gauss--Bonnet, $e(T\F)[\mu]=-1/2\pi$ for this normalization, and  
$$[X]\cap [\mu] = \lim_{i\to \infty} \frac {|X \cap D_i|} {\area(D_i)}$$ 
We deduce the following corollary:
\begin{corollary}[Area inequality]\label{intersection_multiplicity}
Suppose $M$ is atoroidal, and $\phi:(M,\F) \to S^2$ is a generic leafwise branched cover
with simple branch locus $X$. Then for every leaf $\lambda$ in the support of an invariant
transverse measure, and any exhaustion $D_i$ of $\lambda$ by compact sets
with $\length(\partial D_i)/\area(D_i) \to 0$, we have 
$$\lim_{i \to \infty} \frac {|X \cap D_i|} {\area(D_i)} >  \frac {1} {2\pi}$$
as measured in the hyperbolic metric obtained by uniformizing $\lambda$.
\end{corollary}

For any taut foliation $\F$ (transversely measured or not), 
and any transverse 1-manifold $X$ intersecting every leaf, 
then simply from the compactness of $M$ it follows that there are positive constants 
$C$ and $c$ so that for every point $p$ in every leaf $\lambda$ the ball $B_C(p)$
of radius $C$ in $\lambda$ about $p$ must intersect $X$ at least once, and the
ball $B_c(p)$ of radius $c$ must intersect $X$ at most once, so that for sufficiently
big $r$ the intersection numbers $|X\cap B_r(p)|/\area(B_r(p))$ are bounded above and
below by positive constants independent of $r$. 

\subsection{Hopf invariant}

Now suppose $M$ is an oriented homology sphere. For any map $\phi:M \to S^2$ in 
general position, the {\em Hopf invariant} $H(\phi)$ is the linking number of
two generic (suitably) oriented fibers. If $\phi:(M,\F) \to S^2$ is a leafwise branched cover,
we may compute $H(\phi)$ by taking any two points in $S^2$ outside the image of
the branch locus, and orient their preimage by the co-orientation to $\F$.

Let $n \in S^2$ be the north pole; let's assume $n$ is not in the image of the
branch locus. Let $N:=\phi^{-1}(n)$ be the collection of oriented loops in the preimage.
We can think of these as conjugacy classes in $\pi_1(M)$. 

Since $\F$ is taut and $M$ is a homology sphere, there are no invariant transverse
measures, so the leaves are all conformally hyperbolic. Thus
there is a {\em universal circle} $S^1_u$ associated to $\F$, and a representation
$\rho_u:\pi_1(M) \to \text{Homeo}^+(S^1_u)$.
Since $M$ is a homology sphere, there is a
well-defined {\em rotation quasimorphism} $\rot_\F:\pi_1(M) \to \R$. This is a class
function, and therefore well-defined on conjugacy classes, and it extends to a function
on (homotopy classes of) oriented 1-manifolds by additivity.
See e.g.\/ \cite{Calegari_book}
Chapter~7 or \cite{Calegari_Dunfield} for details.

We need to define one more invariant. For any oriented knot $K$ in $M$ transverse to $\F$ the
longitude $\ell$ determines a nonzero section of the torus of unit tangents $UT\F|K$. 
However, this circle bundle also carries a flat ($\text{Homeo}^+(S^1)$ valued) 
connection (up to monotone equivalence) coming
from the universal circle. Relative to the flat connection, $\ell$ has a well-defined
winding number, that we denote $\wind(K)$. For an oriented link define $\wind$ similarly by 
summing over components.

With this notation we have the following:

\begin{theorem}[Hopf invariant formula]
$$H(\phi) = \rot_\F(N) - \wind(N)+ \link(N,X)$$
\end{theorem}
\begin{proof}
This is largely an exercise in unraveling definitions. Let $R$ be a compact oriented
surface immersed in $M$ with oriented boundary $N$, where $\psi:R \to M$ denotes the
immersion. The pullback of $UT\F$ to $R$ is an oriented circle bundle, which has a
flat connection (up to monotone equivalence) coming from the universal circle. 
The longitude $\ell$ gives a trivialization over the boundary, and hence an Euler class
$\psi^*e(T\F) \in H^2(R,\partial R)$, and the evaluation on the fundamental class
$\psi^*e(T\F)[R]$ is equal to the difference $\rot_\F(N) - \wind(N)$.
To see this, we must identify the (bounded) Euler class associated to a flat 
connection on a circle bundle with the coboundary of the rotation quasimorphism;
see e.g.\/ \cite{Ghys_circle} or \cite{Calegari_scl} \S~4.2.4--4.2.5.

Likewise, the composition of $\psi$ with $\phi$ crushes $\partial R$ to the point $n$, so we can
pull back $\psi^*\phi^*e(S^2)$ to a class in $H^2(R,\partial R)$. This class and
$\psi^*e(T\F)$ both pair with the fundamental class $[R]$, and the difference
is equal to the (algebraic) intersection of $R$ with $X$, by relativizing 
Theorem~\ref{Euler_formula}. Furthermore, $[X] \cap [R] = \link(\partial R,X) = \link(N,X)$ 
by definition of linking number in a homology sphere.
\end{proof}

The Milnor--Wood inequality (see \cite{Calegari_scl}) implies
$|\rot_\F(N)| \le -\chi(R)$ where $R$ is any 
compact oriented surface bounding $N$ without disk or sphere components.

If $M$ is merely a rational homology sphere, we obtain a similar formula by replacing
$N$ by a homologically trivial finite cover and dividing through the right hand side by
the degree.

\section{Which Euler classes are realized?}\label{realization}

In this section we ask more pointedly the question:
given $(M,\F)$, which Euler classes $\phi^*e(S^2)$ are realized by leafwise
branched covers $\phi:(M,\F) \to S^2$?
By Theorem~\ref{Euler_formula} this is equivalent to asking what homology
classes $[X]$ can be realized as the (simple) branch locus. 
Say a class $[X] \in H_1(M)$ is {\em realizable} if there is $\phi:(M,\F) \to S^2$ simply
branched over some $X$ in the given class. 

\subsection{Parity Condition} 

Both $e(T\F)$ and $\phi^*e(S^2)$ are integral classes, and therefore so is
$[X]$. These classes are subject to a well-known parity condition:

\begin{proposition}[Parity]\label{parity}
Let $\phi:(M,\F) \to S^2$ be a generic leafwise branched cover. Then the
mod 2 reduction of $[X]$ is zero in $H_1(M;\Z/2\Z)$. Equivalently, $[X]=2s$ for
some $s \in H_1(M;\Z)$.
\end{proposition}
\begin{proof}
Let $[S] \in H_2(M;\Z/2\Z)$ be any homology class.
Let $S$ be a closed immersed surface in $M$ representing $[S]$ (possibly non-orientable).
We need to show $[X] \cap [S]$ is equal to zero mod 2.

Since $\phi^*e(S^2)[S]$ is equal to $\chi(S^2)$ times the (mod 2) 
degree of $\phi:S \to S^2$ it follows that this number is even. Thus
$[X] \cap [S]$ has the same parity as $e(T\F)[S]$.

Now, for any orientable/co-orientable 2-plane field $\xi$ on an orientable 3-manifold
(integrable or not) the Euler class $e(\xi)$ is zero in $H^2(M;\Z/2\Z)$. For, 
the co-orientability of $\xi$ is equivalent to the triviality of the complement
$\xi^\perp$, and therefore the Stiefel-Whitney classes of $\xi$ and of 
$\xi + \xi^\perp = TM$ are equal. 
But every orientable 3-manifold is parallelizable, so $w_2(\xi)=w_2(M)=0$.
As is well-known, $w_2(\xi)\in H^2(M;\Z/2\Z)$ is the mod 2 reduction of 
$e(\xi)$ for an orientable plane bundle.
\end{proof}

By Theorem~\ref{Euler_formula} and Proposition~\ref{parity}, we have the constraints
\begin{enumerate}
\item{$[X]\cap [\mu] > -e(T\F)[\mu]$ for every invariant transverse measure $\mu$; and}
\item{$[X] \cap [S]$ is even for every class $[S] \in H_2(M;\Z/2\Z)$.}
\end{enumerate}

These necessary conditions are not sufficient:

\begin{example}[Closed leaves]
It is especially easy to see these conditions are insufficient when $\F$ contains
closed leaves. For example:
\begin{enumerate}
\item{If $R$ is a closed leaf of $\F$, then $\chi(R)+[X]\cap [R]=2d$ 
where $d$ is the degree of $\phi:R \to S^2$. If $R$ is not a sphere, the degree of
any branched cover must be at least $2$ so $\chi(R)+[X]\cap [R]\ge 4$.}
\item{If $\F$ is a surface bundle and $\phi$ maps every fiber with degree 2, then the
monodromy of $\F$ is conjugate into the hyperelliptic mapping class group. Thus there
are surface bundles with fibers $R$ of every genus $\ge 3$ 
for which $\chi(R)+[X]\cap [R]\ge 6$. }
\end{enumerate}
\end{example}

\subsection{Realizing transversals and plumbing}

Let $Y$ be an oriented 1-manifold winding positively transverse to $\F$ 
and intersecting every leaf.
We call such a $Y$ a {\em total transversal}. A total transversal $Y$ determines
a homology class $[Y]\in H_1(M)$ and it is natural to ask which classes are realized.

Such a $[Y]$ must of course pair positively with every invariant transverse measure
$\mu$ for $\F$. Projectively, this is the only obstruction: for every class
$\alpha \in H_1(M)$ with $\alpha \cap [\mu]>0$ for all $\mu$, there is some positive
integer $n$ and a total transversal $Y$ with $[Y]=n\alpha$. This is proved by the
method of foliation cycles, and is essentially a corollary of the Hahn--Banach
theorem; see e.g.\/ Sullivan \cite{Sullivan_cycles} for definitions
and details.

We observe that the set of $[X]$ realized as the branch locus is closed under
adding even multiples of total transversals:

\begin{proposition}[Plumbing]
Let $\phi:(M,\F) \to S^2$ be a generic leafwise branched covering with simple
branch locus $X$, and let $Y$ be a total transversal. Then there is 
a generic leafwise branch cover $\phi':(M,\F) \to S^2$ with branch locus $X'$
and $[X']=[X]+2[Y]$.
\end{proposition}
\begin{proof}
Plumb $\phi$ along $Y$ leafwise by a double branched $D \to S^2$. This adds two
circles of branch points for each circle of $Y$.
\end{proof}

\subsection{Belyi realization}

Let's analyze the Euler class $\phi^*e(S^2)$ for the Belyi construction from
\S~\ref{Belyi_section}:

\begin{theorem}[Belyi realization]\label{Belyi_euler}
Let $M$ be orientable, let $\F$ be a co-orientable taut foliation,
and let $Y$ be any oriented total transversal.
Then there is a generic leafwise branched map $\phi$ with simple 
branch locus $X$ for which
$$[X]^* = 12\, [Y]^* - 13\, e(T\F)$$
\end{theorem}
\begin{proof}
Apply the Belyi map construction from \S~\ref{Belyi_section} to obtain 
$\phi:(M,\F) \to S^2$. The branch locus is not simple, so we need to count with
multiplicity to determine $X$.

For each leaf $\lambda$ the intersection $\lambda \cap Y$ gives the vertices of the 
initial triangulation. After performing the barycentric subdivision, each triangle of the
original triangulation maps over $S^2$ with degree $3$, and therefore the pullback
$\phi^*e(S^2)[\tau]=6$ for each such triangle $\tau$ (relative to the obvious trivialization
on the boundary). For a closed surface and a triangulation with $t$ triangles and $v$
vertices we have $t/2 = v - \chi$. Thus, for every homology class $[S]$,
$$\phi^*e(S^2)[S] = 12\,[Y] \cap [S] - 12\, e(T\F)[S]$$
and therefore
$$[X]^*= 12\,[Y]^* - 13\, e(T\F)$$
as claimed.
\end{proof}

In the Belyi construction we have traded efficiency for clarity. 
Thus the constants $12$ and $-13$
in this theorem are far from optimal; see \S~\ref{sketch}. The point is simply to obtain
constants independent of the triple $(M,\F,Y)$.

\begin{remark}
Ghys \cite{Ghys}, Thm.~7.5 shows for $\F$ with conformally hyperbolic leaves 
that for any transversal
$Y$ positively intersecting every leaf, there is {\em some} positive $n$ so that 
the line bundle $2K$ of leafwise quadratic differentials has meromorphic sections 
with poles of order at most $n$ along $Y$. In other words, there are (many) holomorphic
sections of the line bundle $2K+n[Y]$. If $L$ is a holomorphic
line bundle, the ratio $f/g$ of two holomorphic sections $f,g$ defines a holomorphic
map $f/g$ to $\CP^1$, and the pullback $(f/g)^*T\CP^1 = 2L$. Thus, for the maps
$\phi:(M,\F) \to S^2$ that Ghys constructs, $\phi^*e(S^2)=2n[Y]^*-4\,e(T\F)$ and
$[X]^* = 2n[Y]^*-5\,e(T\F)$.
\end{remark}

\subsection{Sketch of a construction}\label{sketch}

We sketch a construction to improve the constants in Theorem~\ref{Belyi_euler}.
Fix $(M,\F)$. Given a total transversal $Y$, we let $2Y$ denote 2 parallel copies of $Y$,
and label these copies $Y_0,Y_1$. We shall construct leafwise triangulations 
and transitions between them of bounded combinatorial complexity 
for which each triangle has two vertices on each of $Y_0,Y_1$. 

Here's how to do this. In each leaf $\lambda$ we have a two discrete sets of points with labels
$0,1$, each of which is a separated net. The first step is to construct an embedded
directed graph $\Gamma$ going from $0$ vertices to $1$ vertices. Join every $0$ vertex by
shortest oriented geodesics to all the $1$ vertices within some (bounded) distance $C$. 
When two oriented edges cross, we resolve the crossing in an orientation-preserving way and
straighten the result (this might collapse some previously distinct edges). We can perform
all such resolutions simultaneously and obtain in this way 
an embedded geodesic directed graph $\Gamma$ 
connecting $0$ vertices to $1$ vertices. The complement $\lambda - \Gamma$ might have 
nontrivial topology, but it has uniformly bounded thickness: every point is within
a uniformly bounded distance of $\Gamma$. Thus we may add geodesic 
directed edges between $0$ and $1$ vertices on $\Gamma$ to produce a new $\Gamma'$ for
which $\lambda - \Gamma'$ consists entirely of polygonal disks (of bounded valence).
Add the shortest directed geodesic edge at each stage (just for definiteness),
resolving crossings (if any) if necessary. If we add maximally many edges to 
produce $\Gamma'$, complementary regions will all be quadrilaterals. 
Put a vertex in the center of each quadrilateral and triangulate.

The resulting triangulation has a canonical Belyi map to $S^2$, by making the map
orientation-preserving on each triangle, and taking vertices
of $\lambda \cap Y_0$ and $\lambda \cap Y_1$ to $0$ and $1$.
If we are judicious, this construction can be done canonically (and hence continuously
from leaf to leaf) except at finitely 
many isolated places where a hexagon is decomposed in two different ways into two
quadrilaterals. Each decomposition determines a map from the hexagon to $S^2$ of
degree 4, and we can cut open and insert an isotopy of degree 4 maps interpolating
between them.

Every quadrilateral maps to $S^2$ with degree $2$, so $\phi^*e(S^2)[Q]=4$ for each
quadrilateral $Q$. A decomposition of a closed surface into $q$ quadrilaterals and $v$ vertices 
has $q = v-\chi$. The vertices leafwise represent $2Y$.
Thus $\phi^*e(S^2) = 8[Y]^*-4e(T\F)$ and $[X]^* = 8[Y]^*-5e(T\F)$.

\medskip

What would it take to move beyond this construction? Given a taut foliation $\F$ and
a total transversal $Y$, our goal is to realize a simple
branch locus $X$ satisfying $[X]^* = \alpha [Y]^* - \beta e(T\F)$ for 
$(\beta,\alpha)$ as small as possible in the lexicographic ordering. If $\alpha$
and $\beta$ are rational numbers (or even if $\beta$ is an integer and $\alpha$ is
an odd integer), this imposes divisibility conditions on
the classes $[Y]$ and $e(T\F)$, since $[X]^*$ is always integral and even.

Here is a modest attempt. Suppose we have a leafwise decomposition into quadrilaterals 
with vertices at $2Y$, as above. A {\em matching} is a pairing of quadrilaterals,
where each quadrilateral is paired with an adjacent one. The existence of a matching
implies that $e(T\F)$ and $2[Y]$ are even, but fortunately this condition is
always satisfied. Also necessary is to choose a matching which varies continuously
as a function of leaves, except at finitely many isolated locations.

Given a matching we can remove the edges of $\Gamma'$ separating matched quadrilaterals,
and produce a new graph $\Gamma''$ with vertices at $2Y$ whose adjacent regions are
all hexagons. In the associated Belyi map every hexagon maps to $S^2$ with degree $3$, 
and $2h = v - \chi$ so $\phi^*e(S^2) = 6[Y]^* - 3e(T\F)$. 

When does a leafwise matching exist? For closed leaves we can always find such a
matching. Also we can find a matching for leaves with exponential growth, since 
Hall's Marriage theorem lets us extend any given partial matching locally. 
For leaves of polynomial growth --- precisely those that have the support of
nontrivial invariant measures in their closures --- this combinatorial condition is
more subtle, and perhaps can't be solved in general for arbitrary $Y$.

\section{Acknowledgments}
I would like to thank Ian Agol, Lvzhou Chen, Nathan Dunfield, Benson Farb, 
Andras Juhasz, Pablo Lessa, Cliff Taubes and Mehdi Yazdi for useful feedback.

\end{document}